\documentclass{amsart}

\usepackage{amsmath}
\usepackage{amsthm}
\usepackage{amssymb}
\usepackage{enumerate}

\newtheorem{thm}{Theorem}[section]
\newtheorem{theorem}{Theorem}[section]
\newtheorem{prop}[thm]{Proposition}

\newtheorem{lemma}[thm]{Lemma}
\newtheorem{cor}[thm]{Corollary}

\theoremstyle{definition}

\theoremstyle{remark}

\numberwithin{equation}{section}

\newcommand{\R}{\mathbb{R}}  
\newcommand\numberthis{\addtocounter{equation}{1}\tag{\theequation}}

\everymath{\displaystyle}
\allowdisplaybreaks[1]

\begin{document}


\title{Simplicity and Finiteness of Discrete Spectrum of the Benjamin-Ono Scattering Operator}


\author{Yilun Wu}
\address{Mathematics Department, Brown University, Providence, RI 02912}
\email{yilunwu@umich.edu}
\urladdr{http://www-personal.umich.edu/~yilunwu/} 





\begin{abstract}
A spectral analysis is done on the $L$ operator of the Lax pair for the Benjamin-Ono equation. Simplicity and finiteness of the discrete spectrum are established as are needed for the Fokas and Ablowitz inverse scattering transform scheme. A crucial step in the simplicity proof is the discovery of a new identity connecting the $L^2$ norm of the eigenvector to its inner product with the scattering potential. The proof for finiteness is an extension of the ideas involved in the Birman-Schwinger bound for Schr\"odinger operators.
\end{abstract}


 \maketitle




\bibliographystyle{acm}

\section{Introduction}

The Benjamin-Ono equation is 
\begin{equation}\label{eq: BO}
u_t +2uu_x - Hu_{xx} = 0,
\end{equation}
where the Hilbert transform $H$ is defined as
\begin{equation}
H\varphi(x) = \text{P.V.}\frac{1}{\pi}\int_{-\infty}^{\infty}\frac{\varphi(y)}{x-y}~dy.
\end{equation}
Notice the opposite sign also appears in the literature when defining the Hilbert transform. 

\eqref{eq: BO} was first formulated by Benjamin \cite{benjamin1967internal} and Ono \cite{ono1975algebraic}. It can be used to model internal gravity waves in a two layer fluid. See also Davis and Acrivos \cite{davis1967solitary}, Choi \cite{choi1999fully} and and Xu \cite{xu2010asymptotic} for more details on the derivation of \eqref{eq: BO}. Applications of \eqref{eq: BO} include internal wave motions supported by oceanic thermoclines and atmospheric waves.

Local and global well-posedness of \eqref{eq: BO} were obtained by Saut \cite{saut1979sur}, I{\'o}rio \cite{jose1986cauchy}, Ponce \cite{ponce1991global}, Koch and Tzvetkov \cite{koch2003local}, Kenig and Koenig \cite{kenig2003local} and Tao \cite{tao2004global}. \eqref{eq: BO} is further found to be a completely integrable equation. The Lax pair of \eqref{eq: BO} was discovered by Nakamura \cite{nakamura1979direct} and Bock and Kruskal \cite{bock1979two}. Fokas and Ablowitz \cite{fokas1983inverse} formulated the direct and inverse scattering transforms for \eqref{eq: BO} and obtained soliton solutions. See also Kaup and Matsuno \cite{kaup1998inverse} and Xu \cite{xu2010asymptotic}. As is proven for other integrable equations, the inverse scattering transform method (IST) can be applied powerfully to all kinds of asymptotic and stability analyses. For instance, there is an extensive literature on long time asymptotics and zero dispersion asymptotics of the Korteweg-de Vries equation by the IST. Here we just mention a few rigorous works. The long time asymptotics of the Korteweg-de Vries equation were obtained by many earlier works and were justified by Shabat \cite{shabat1973korteweg}, Tanaka \cite{tanaka1975korteweg}, Buslaev and Sukhanov \cite{buslaev1986asymptotic}, Deift and Zhou \cite{deift1993steepest} and others. The zero dispersion aysmptotics were established by Lax and Levermore \cite{lax1979zero, lax1983small}, Venakides \cite{venakides1985zero, venakides1990korteweg}, Deift, Venakides and Zhou \cite{deift1997new} and others. In comparison, there is only a limited amount of work on the asymptotic analysis of the Benjamin-Ono equation. In particular, Miller and Xu \cite{miller2011zero} were able to establish a zero dispersion limit for generic positive initial data. The relative scarcity of asymptotic analysis for the Benjamin-Ono equation is partly due to the fact that the IST for the Benjamin-Ono equation with general initial data is still to a large extent formal and needs to be understood more thoroughly. Coifman and Wickerhauser \cite{coifman1990scattering} first did a rigorous analysis of this problem. They were able to justify an IST related to the Fokas and Ablowitz scheme for small potential and obtained small data global well-posedness of \eqref{eq: BO} in that way. Their analysis, however, was very complicated and did not closely follow the Fokas and Ablowitz IST. They also avoided a study of discrete spectrum due to the small data assumption. Up to the present time, a rigorous analysis of the Fokas and Ablowitz IST without a small data assumption is still lacking, and there is no global well-posedness result by IST for large data. As a first step toward that problem, this paper studies the discrete spectrum of the direct scattering problem, and establishes a few key properties that are useful for a construction of the scattering data in the Fokas and Ablowitz IST.

A nice way to look at the Lax pair of \eqref{eq: BO} is that it essentially decomposes with respect to the ranges of $C_{\pm}$ where 
\begin{equation}
C_{\pm}\varphi = \frac{\varphi \pm iH\varphi}{2}
\end{equation}
are the Cauchy projections. When $C_{\pm}$ act on $L^2(\R)$, the ranges are $H^{\pm}$: the Hardy spaces of $L^2$ functions whose Fourier transforms are supported on the positive and negative half lines. In this paper we adopt the following convention for the Fourier and the inverse Fourier transforms:
\begin{equation}F(f)(x)=\hat{f} (\xi) = \int_{\mathbb{R}}e^{-ix\xi}f(x)~dx,\end{equation}
\begin{equation}F^{-1}(f)(x) = \check{f}(x) = \frac{1}{2\pi}\int_{\mathbb{R}}e^{i\xi x}f(\xi)~d\xi.\end{equation}
One then has
\begin{equation}\widehat{C_{\pm}\varphi}= \chi_{\mathbb{R}_{\pm}}\hat{\varphi}.\end{equation}
Notice $C_{\pm}$ acts as identity on $H^{\pm}$ respectively.
The Lax pair is described as follows. On $H^+$, we have
\begin{align}
L_u\varphi &= \frac{1}{i}\varphi_x -C_+(uC_+\varphi)\\
B_u\varphi &= \frac{1}{i}\varphi_{xx} +2[(C_+u_x)(C_+\varphi)-C_+(u_xC_+\varphi)-C_+(uC_+\varphi_x)],
\end{align} 
and on $H^-$, we have
\begin{align}
L_u\varphi &= -\frac{1}{i}\varphi_x -C_-(uC_-\varphi)\\
B_u\varphi &= -\frac{1}{i}\varphi_{xx} +2[(C_-u_x)(C_-\varphi)-C_-(u_xC_-\varphi)-C_-(uC_-\varphi_x)].
\end{align} 
The presentation of the Lax pair here is apparently different from those shown in the literature. To derive this Lax pair, one notes the following formal identity:
\begin{equation}\label{eq: cauchy proj iden}
C_{\pm}(fg) + (C_{\pm}f)(C_{\pm}g)-C_{\pm}(fC_{\pm}g)-C_{\pm}(gC_{\pm}f)=0.
\end{equation}
\eqref{eq: cauchy proj iden} is true for either the ``$+$" sign or the ``$-$" sign, as can be shown by taking the Fourier transform. A consequence of \eqref{eq: cauchy proj iden} is the following:
\begin{equation}\label{eq: cauchy proj iden 2}
C_{\pm}((C_{\pm}f)(C_{\pm}g)) = (C_{\pm}f)(C_{\pm}g).
\end{equation}
Let us first look at the Lax pair on $H^+$. The commutator of $L_u$ with $B_u$ is
\begin{equation}\label{eq: lax pair deri 1}
[L_u,B_u] = \left[\frac{1}{i}\partial_x - C_+uC_+, \frac{1}{i}\partial_x^2 + 2[(C_+u_x)C_+-C_+u_xC_+ -C_+uC_+\partial_x)]\right].
\end{equation}
One can evaluate the various terms in the commutator one by one to get
\begin{equation}\label{eq: lax pair deri 2}
[L_u,B_u] = \frac{2}{i}(C_+u_{xx})C_+ - \frac{1}{i}C_+u_{xx}C_+ - 2[C_+uC_+(C_+u_x)C_+ - (C_+u_x)C_+uC_+ + C_+u_xC_+uC_+].
\end{equation}
If we let $[L_u,B_u]$ act on $\varphi$, the terms in the square brackets in \eqref{eq: lax pair deri 2} reads
\begin{align}
& C_+(uC_+((C_+u_x)(C_+\varphi)))-(C_+u_x)(C_+(uC_+\varphi))+C_+(u_xC_+(uC_+\varphi))\notag\\
= ~&C_+((C_+u_x)(uC_+\varphi))-(C_+u_x)(C_+(uC_+\varphi))+C_+(u_xC_+(uC_+\varphi))\notag\\
= ~&C_+(u_xuC_+\varphi) = C_+uu_xC_+(\varphi),
\end{align}
where we use \eqref{eq: cauchy proj iden 2} in the first step and \eqref{eq: cauchy proj iden}  with $f=u_x$ and $g=uC_+\varphi$ in the second step. Therefore
\begin{equation}
[L_u,B_u] = C_+\left(\frac{1}{i}(2C_+u_{xx}- u_{xx}) -2uu_x\right)C_+ = C_+(Hu_{xx} - 2uu_x)C_+,
\end{equation}
and 
\begin{equation}
\partial_tL_u + [L_u,B_u] = -C_+(u_t+2uu_x-Hu_{xx})C_+.
\end{equation}
The situation on $H^-$ is similar. One gets
\begin{equation}
\partial_tL_u + [L_u,B_u] = -C_-(u_t+2uu_x-Hu_{xx})C_-.
\end{equation}
Therefore \begin{equation}\label{eq: Lax eq} \partial_t L_u + [L_u,B_u] = 0\end{equation} if \eqref{eq: BO} is satisfied. On the other hand, if \eqref{eq: Lax eq} is satisfied, one obtains 
\begin{equation}\label{eq: vanishing f 1}
C_+( f C_+ \varphi )= C_- (f C_-\varphi )= 0
\end{equation}
for every $\varphi$, say, in the Schwartz class, where $f=u_t+2uu_x - Hu_{xx}$. We now show that $f$ must be zero, i.e., \eqref{eq: BO} is satisfied. As a formal argument, we assume that $f$ belongs to the Schwartz class, but the argument can be easily extended to other function spaces by approximation. In fact, the Fourier transform of \eqref{eq: vanishing f 1} gives
\begin{equation}\label{eq: vanishing f 2}
\int_{-\infty}^\xi \hat{f}(\eta)\hat{\varphi}(\xi-\eta)~d\eta = 0, \quad \xi\ge 0,
\end{equation}
and
\begin{equation}\label{eq: vanishing f 3}
\int_{\xi}^{\infty} \hat{f}(\eta)\hat{\varphi}(\xi-\eta)~d\eta = 0, \quad \xi\le 0,
\end{equation}
Take $\xi = 0$ and add up \eqref{eq: vanishing f 2} and \eqref{eq: vanishing f 3} to get
\begin{equation}\label{eq: vanishing f 4}
\int_{\mathbb{R}}\hat{f}(\eta)\hat{\varphi}(-\eta)~d\eta=0.
\end{equation}
Since \eqref{eq: vanishing f 4} holds for every $\varphi$, $f$ must be zero.

The Cauchy projections can be extended to act on larger spaces when necessary. Notice that when $u$ is real, the Lax pair on $H^-$ is just the complex conjugate of the Lax pair on $H^+$. In the following, we will always assume $u$ to be real, and will focus on the $H^+$ part of the Lax pair. The scattering data of the IST are closely related to the spectrum of the operator $L_u = \frac{1}{i}\partial_x - C_+uC_+$. Notice that $\frac{1}{i}\partial_x$ is self-adjoint. Hence $L_u$ can be regarded as a perturbation of $\frac{1}{i}\partial_x$ and is also self-adjoint (see \cite{miller2011zero}):

\begin{theorem}\label{thm: 1}
Suppose $u\in L^2(\R)\cap L^{\infty}(\R)$. $L_u$ is a relatively compact perturbation of $\frac{1}{i}\partial_x$ and is self-adjoint on $H^+$ with domain $H^+\cap H^1(\R)$. Here $H^1(\R)$ is the $L^2$ Sobolev space.
\end{theorem}

\cite{miller2011zero} did not mention a proof to this result, so we provide a proof in Section \ref{sec: simp}. As a consequence of Theorem \ref{thm: 1}, by Weyl's theorem for the spectrum of self-adjoint operators, the essential spectrum of $L_u$ is the same as that of $\frac{1}{i}\partial_x$, i.e. $\R^+\cup \{0\}$. By general spectral theory (see \cite{reed1980methods}), the negative discrete spectrum consists of isolated eigenvalues of finite multiplicity. However, it is not clear whether the eigenvalues are simple, or whether there are only finitely many of them. Yet these pieces of information are crucial for the construction of scattering data in the Fokas and Ablowtiz IST. The following Theorem is the main result of this paper:

\begin{theorem}\label{thm: 2}
Suppose $u\in L^1(\R)\cap L^{\infty}(\R)$ and $xu(x)\in L^2(\R)$. The operator $L_u$ has only finitely many negative eigenvalues, and the dimension of each eigenspace is 1.
\end{theorem}

Finiteness is proven in Section \ref{sec: fin}, and simplicity is proven in Section \ref{sec: simp}. We make a few remarks on the conditions on $u$. The first condition on $u$ implies $u$ belongs to every $L^p$ space, for $1\le p \le \infty$. Furthermore, we could change the $L^1$ condition to $L^2$ and would have made no difference. In fact, if $u\in L^2(\R)$ and $xu(x)\in L^2(\R)$, then
\begin{equation}\int_{\R}|u(x)|~dx = \int_{\R}\sqrt{1+|x|^2}|u(x)|\frac{1}{\sqrt{1+|x|^2}}~dx\le C\left(\int_{\R}|(1+|x|^2)||u(x)|^2~dx\right)^{1/2}<\infty.\end{equation}

The next step in the spectral analysis of $L_u$ will be a study of the essential spectrum $[0, \infty )$. For $\lambda$ in this region one wants to be able to solve certain integral equations related to the eigenvalue problem, and it will be of interest to show that the kernel of $L_u-\lambda$ is trivial. In other words, there is no embedded eigenvalues in the essential spectrum. This is a task yet to be accomplished.

\section{Simplicity of the Discrete Spectrum}\label{sec: simp}

In this section, we establish some basic properties of $L_u$ and then prove simplicity of its negative eigenvalues. All $L^p$ spaces  in the following are understood to have domain $\R$ if their domains are unspecified. 
We first show Theorem \ref{thm: 1}. 

\begin{proof}[Proof of Theorem \ref{thm: 1}]
We just need to show that \begin{equation}C_+uC_+ (-i\partial_x-i)^{-1}: H^+ \to H^+\end{equation} is compact. Let us conjugate this operator by the Fourier transform to get \begin{equation}FC_+ u C_+ (-i\partial_x-i)^{-1}F^{-1}: L^2(0,\infty) \to L^2(0,\infty).\end{equation} A simple calculation shows that this is an integral operator with kernel \begin{equation}K(x,y)=\chi_{\R^+}(x)\chi_{\R^+}(y)\frac{\hat{u}(x-y)}{y-i}.\end{equation}
Since \begin{equation}\int_{\R^2}\chi_{\R^+}(x)\chi_{\R^+}(y)\frac{|\hat{u}(x-y)|^2}{|y-i|^2}~dx~dy\le C\|u\|_2^2,\end{equation}
the operator in question is Hilbert-Schmidt, hence is compact. Therefore $C_+uC_+$ is $-i\partial_x$-compact, hence $-i\partial_x$-bounded with relative bound 0. By the Kato-Rellich theorem (see \cite{reed1975fourier}, Theorem X.12), $L_u$ is self adjoint on $H^+\cap H^1$, because $C_+uC_+$ is symmetric.
\end{proof}

\begin{lemma}\label{lem: semi-boundedness}
Suppose $u\in L^2 \cap L^{\infty}$. For any $\varphi \in D(L_u)$, \begin{equation}(\varphi, L_u\varphi)\ge -\|u\|_{\infty}\|\varphi\|_2^2.\end{equation} Here $D(L_u)$ is the domain of $L_u$.
\end{lemma}
\begin{proof}
By taking the Fourier transform, one easily sees that
\begin{equation}(\varphi,L_u\varphi)\ge -(C_+uC_+\varphi,\varphi),\end{equation}
whereas
\begin{equation}|(C_+uC_+\varphi,\varphi)|=|(uC_+\varphi,C_+\varphi)|\le \|u\|_{\infty}\|\varphi\|_2^2.\end{equation}
\end{proof}

From Lemma \ref{lem: semi-boundedness} we see that the negative eigenvalues are bounded from below. In fact, all of them are no less than $-\|u\|_{\infty}$.

The following Lemma allows us to convert between the differential and the integral form of the eigenvalue equation. It provides some details to the equations that appeared in \cite{fokas1983inverse}. The convolution kernel $G_{\lambda}$ in the integral equations are given by 
\begin{equation}
G_{\lambda}(x) = \frac{1}{2\pi}\int_0^{\infty}\frac{e^{ix\xi}}{\xi-\lambda}~d\xi.
\end{equation}

\begin{lemma}[Equivalence of differential and integral equations]\label{lem: equiv diff int eq}
Suppose $u\in L^2\cap L^{\infty}$. 
\begin{enumerate}
\item 
If $\varphi\in H^+\cap H^1$, $\lambda<0$, then
\begin{equation}\label{eq: e-value diff eq}
\frac{1}{i}\partial_x\varphi - C_+(u\varphi) = \lambda \varphi
\end{equation}
if and only if
\begin{equation}\label{eq: e-value int eq}
\varphi = G_{\lambda}*(u\varphi).
\end{equation}
\item 
If $\varphi-1\in H^+\cap H^1$, $\lambda<0$, then
\begin{equation}
\frac{1}{i}\partial_x(\varphi-1) - C_+(u(\varphi-1)) = \lambda (\varphi-1)+C_+u
\end{equation}
if and only if
\begin{equation}
\varphi = 1+G_{\lambda}*(u\varphi).
\end{equation}
\end{enumerate}
Any $\varphi$ satisfying the above equations is bounded and continuous.
\end{lemma}
\begin{proof}
We just show the equivalence of \eqref{eq: e-value diff eq} and \eqref{eq: e-value int eq}. The other case is similar. Let us Fourier transform \eqref{eq: e-value diff eq}:
\begin{equation}\xi\hat{\varphi}(\xi)-\chi_{\R^+}(\xi)\widehat{u\varphi}(\xi)=\lambda \hat{\varphi}(\xi),\end{equation}
\begin{equation}\hat{\varphi} = \frac{\chi_{\R^+}(\xi)}{\xi-\lambda}\widehat{u\varphi}(\xi).\end{equation}
Now inverse Fourier transform to get \eqref{eq: e-value int eq}. One just needs to notice that $\frac{\chi_{\R^+}(\xi)}{\xi-\lambda}\in L^2$ to make the argument rigorous. Finally, by the Sobolev embedding theorem, $H^1(\R)\subset C^{0,\frac{1}{2}}(\R)$. Hence $\varphi$ is bounded and continuous.
\end{proof}

\begin{lemma}
If $\lambda<0$, the Hilbert-Schmidt norm of $G_{\lambda}*(u~\cdot): H^+\to H^+$ is $\frac{\|u\|_2}{\sqrt{|\lambda|}}$.
\end{lemma}
\begin{proof}
The $L^2$ norm squared of the kernel can be computed as follows:
\begin{align}
& \int_{\R^2}|G_\lambda(x-y)|^2 |u(y)|^2~dx~dy \notag \\
= &\|u\|_2^2\int_{\R}|G_\lambda(x)|^2~dx \notag\\
= &\|u\|_2^2\int_0^{\infty}\frac{1}{(\xi-\lambda)^2}~d\xi \notag\\
= &\frac{\|u\|_2^2}{ |\lambda|}.
\end{align}
\end{proof}


\begin{lemma}\label{lem: asymp}
If $u\varphi\in L^{1}\cap L^2$ and $xu(x)\varphi(x)\in L^2$, \begin{equation}G_{\lambda}*(u\varphi)(x) = \frac{1}{2\pi i \lambda x}\int_{\mathbb{R}}u(y)\varphi(y)~dy + \frac{R(\lambda,x)}{x} \quad \text{for } x\ne 0,\end{equation} where $R(\lambda,x)\in L^2$ for any given $\lambda<0$.
\end{lemma}
\begin{proof}
\begin{equation}G_{\lambda}(x-y)=\frac{1}{2\pi}\int_0^{\infty} \frac{e^{i(x-y)\xi}}{\xi-\lambda}~d\xi=\lim _{n\to \infty}\frac{1}{2\pi}\int_0^n \frac{e^{i(x-y)\xi}}{\xi-\lambda}~d\xi.\end{equation}
The convergence is in the $L^2$ sense in $y$ for each fixed $x$. Therefore
\begin{align}
&G_{\lambda}*(u\varphi)(x)\notag\\
= &\lim_{n\to \infty}\int_{\R}\frac{1}{2\pi}\int_0^n \frac{e^{i(x-y)\xi}}{\xi-\lambda}~d\xi ~u(y)\varphi(y)~dy\notag\\
= &\frac{1}{2\pi}\lim_{n\to \infty}\int_0^n \frac{e^{ix\xi}}{\xi-\lambda} \int_{\R} e^{-iy\xi}u(y)\varphi(y)~dy~d\xi.
\end{align}
Let $g(\xi) = \int_{\R} e^{-iy\xi}u(y)\varphi(y)~dy$. By the conditions given, we see that $g\in H^1$ and $\lim_{\xi \to \infty}g(\xi)=0$. We then have
\begin{align}
&G_{\lambda}*(u\varphi)(x)\notag\\
= &\frac{1}{2\pi}\left(\frac{1}{ix}e^{ix\xi}\frac{g(\xi)}{\xi-\lambda}\bigg|_0^{\infty}-\lim_{n\to \infty}\int_0^n \frac{e^{ix\xi}}{ix} \left(\frac{g'(\xi)}{\xi-\lambda}-\frac{g(\xi)}{(\xi-\lambda)^2}\right)~d\xi\right)\notag\\
= &\frac{g(0)}{2\pi i \lambda x}-\frac{1}{2\pi ix}\int_0^{\infty} e^{ix\xi}\left(\frac{g'(\xi)}{\xi-\lambda}-\frac{g(\xi)}{(\xi-\lambda)^2}\right)~d\xi.
\end{align}
The use of the fundamental theorem of calculus above can be justified by approximating the $H^1$ function $g$ by smooth functions and applying the Sobolev embedding theorem when taking the limit. It is now sufficient to show that $\frac{g'(\xi)}{\xi-\lambda}-\frac{g(\xi)}{(\xi-\lambda)^2}\in L^2$, but $u\varphi\in L^1$ implies $g\in L^{\infty}$ and $xu(x)\varphi(x)\in L^2$ implies $g'\in L^2$. 
\end{proof}

Equation \eqref{eq: mystery eq} in the following Lemma is an important ingredient in the proof of simplicity of the discrete spectrum.

\begin{lemma}\label{lem: interesting indentity}
Assume $u\in L^2\cap L^{\infty}$ and $xu(x) \in L^2$. If $\lambda<0$ is an eigenvalue of $L_u$, and $\varphi$ is an eigenvector, then
\begin{equation}\label{eq: mystery eq}
\bigg|\int_{\mathbb{R}}u\varphi~dx \bigg|^2= 2\pi|\lambda|\int_{\mathbb{R}}|\varphi|^2~dx.
\end{equation}
\end{lemma}
\begin{proof}
Since $\varphi$ is an eigenvector of $L_u$, have 
\begin{equation}\frac{1}{i}\partial_x\varphi - C_+(u\varphi) = \lambda \varphi.\end{equation} 
Fourier transform to get
\begin{equation}\widehat{u\varphi}\chi_{\mathbb{R}^+}(\xi) = (\xi-\lambda)\hat{\varphi}.\end{equation}
In other words
\begin{equation}\label{eq: 1}
\widehat{u\varphi} = (\xi-\lambda)\hat{\varphi}
\end{equation}
when $\xi > 0$. Since $xu(x)\in L^2$ and $\varphi \in L^{\infty}$ by Lemma \ref{lem: equiv diff int eq} , from \eqref{eq: 1} we see that $\hat{\varphi}$ is in $H^1$ on $(0,\infty)$, continuous on $[0, \infty)$, and $\hat{\varphi}\to 0$ as $\xi\to \infty$ because $u\varphi \in L^1$.
Consider
\begin{align}
&|\hat{\varphi}|^2 - \frac{\lambda}{2}\frac{d}{d\xi}|\hat{\varphi}|^2\notag\\
= &\frac{1}{2}[(\hat{\varphi}-\lambda\hat{\varphi}')\bar{\hat{\varphi}} +(\bar{\hat{\varphi}}-\lambda\bar{\hat{\varphi}}')\hat{\varphi}]\notag\\
= &\frac{1}{2}[(\widehat{u\varphi}'-\xi\hat{\varphi}')\bar{\hat{\varphi}} +(\overline{\widehat{u\varphi}}'-\xi\bar{\hat{\varphi}}')\hat{\varphi}].
\end{align}
The last equality follows from the derivative of \eqref{eq: 1}. We now integrate, remembering that $\hat{\varphi}$ is supported on $\mathbb{R}^+$:
\begin{align}
&\int_0^{\infty}|\hat{\varphi}|^2~d\xi + \frac{\lambda}{2}|\hat{\varphi}(0)|^2\notag\\ 
= &\frac{1}{2}\int_0^{\infty} [(\widehat{u\varphi}'-\xi\hat{\varphi}')\bar{\hat{\varphi}} +(\overline{\widehat{u\varphi}}'-\xi\bar{\hat{\varphi}}')\hat{\varphi}]~d\xi\notag\\ 
= &\frac{1}{2}\bigg(-\int_0^{\infty}\xi (|\hat{\varphi}|^2)'~d\xi + \int_{\mathbb{R}}(\widehat{u\varphi}'\bar{\hat{\varphi}} + \overline{\widehat{u\varphi}}'\hat{\varphi})~d\xi\bigg).\label{eq: vanishing second term}
\end{align}
We claim that the second term in \eqref{eq: vanishing second term} vanishes. In fact, by Plancherel's identity, 
\begin{align}
& \int_{\mathbb{R}}(\widehat{u\varphi}'\bar{\hat{\varphi}} + \overline{\widehat{u\varphi}}'\hat{\varphi})~d\xi\notag\\
= &\int_{\R}(-ixu(x)\varphi(x)\overline{\varphi(x)} + \overline{-ixu(x)\varphi(x)}\varphi(x))~dx\notag\\
= &~0.
\end{align}
The first term in \eqref{eq: vanishing second term} evaluates to
\begin{equation}\frac{1}{2}\int_0^{\infty}|\hat{\varphi}|^2~d\xi.\end{equation}
Hence we have
\begin{equation}\int_0^{\infty}|\hat{\varphi}|^2~d\xi = -\lambda|\hat{\varphi}(0)|^2.\end{equation}
By \eqref{eq: 1}, $\widehat{u\varphi}(0) = -\lambda\hat{\varphi}(0)$. Therefore
\begin{equation}-\lambda\int_0^{\infty}|\hat{\varphi}|^2~d\xi = |\widehat{u\varphi}(0)|^2,\end{equation}
which by the Plancherel identity is nothing but the claim.
\end{proof}

We are now ready to show 
\begin{prop}\label{prop: simplicity}
Assume $u\in L^2\cap L^{\infty}$ and $xu(x)\in L^2$. The dimension of eigenspace of every eigenvalue of $L_u$ in the discrete spectrum is 1.
\end{prop}
\begin{proof}
Suppose $\lambda<0$ is an eigenvalue. Denote $G_{\lambda}*(u~\cdot)$ by $T$. One can easily check that $T$ is compact. By Lemma \ref{lem: equiv diff int eq}, $\varphi$ is an eigenvector with eigenvalue $\lambda$ if and only if $\varphi\in \text{Ker}(I-T)$. We denote the dimension of the $\lambda$ eigenspace by $N$. $N$ is certainly finite, either by general spectral theory, or by the fact that $T$ is compact. The idea of the proof is described as follows. We first use Lemma \ref{lem: interesting indentity} to evaluate the singular part of the resolvent of $L_u$ around $\lambda$, and use that with the integral form of the resolvent equation to eventually arrive at \eqref{eq: V sat eq} which roughly says that Ran$(I-T)$ must be suitably``large" if $N$ is greater than one. On the other hand, by directly computation one can see that Ker$(I-T^*)$ must also be suitably ``large". A contradiction will then be formed since Ran$(I-T)$ and Ker$(I-T^*)$ are perpendicular to each other. Let us now provide the details. Pick an orthogonal basis $\{\phi_j\}$ of the eigenspace, and let each $\phi_j$ be normalized so that
\begin{equation}\label{eq: normal} \int_{\mathbb{R}}u\phi_j ~dx = 2\pi i \lambda.\end{equation} 
Note that the left hand side is non-zero by Lemma \ref{lem: interesting indentity}. Let $R(\zeta)$ be the resolvent of $L_u$ at $\zeta$, we have in a neigborhood of $\lambda$ (see \cite{kato1995perturbation}, Chapter 5, Section 3.5):
\begin{equation}R(\zeta) = -\frac{P}{\zeta-\lambda} + h(\zeta).\end{equation}
Here $P$ is projection onto the eigenspace, and $h$ is holomorphic at $\lambda$. Let $W(\zeta)-1 \in H^+\cap H^1$ solve
\begin{equation}\label{eq: 2}
W = 1+G_{\zeta}*(uW).
\end{equation}
$W$ is the Jost solution of Fokas and Ablowtiz (see \cite{fokas1983inverse}).
By Lemma \ref{lem: equiv diff int eq}, Lemma \ref{lem: interesting indentity}, and the normalization of the $\phi_j$'s, we have
\begin{align}
W(\zeta) -1 = &R(\zeta)C_+u \notag\\
= &-\frac{1}{\zeta-\lambda}\sum_{j=1}^N\bigg(\frac{\int_{\mathbb{R}} u \overline{\phi_j}~dx}{\|\phi_j\|}\bigg)\frac{\phi_j}{\|\phi_j\|} + H(\zeta)\notag\\
= &-\frac{1}{\zeta-\lambda}\sum_{j=1}^N\bigg(\frac{-2\pi \lambda \phi_j}{\int_{\mathbb{R}}u\phi_j~dx}\bigg) + H(\zeta)\notag\\
= &\frac{-i}{\zeta-\lambda}\sum_{j=1}^N\phi_j + H(\zeta). \label{eq: W decomp}
\end{align}
Of course $H(\zeta)=h(\zeta)C_+u$. We now plug \eqref{eq: W decomp} into \eqref{eq: 2} to get
\begin{equation}H(\zeta) - \frac{i}{\zeta-\lambda}\sum_{j=1}^N\phi_j =  G_{\zeta}*\bigg(u\big(H(\zeta) - \frac{i}{\zeta-\lambda}\sum_{j=1}^N\phi_j+1\big)\bigg),\end{equation}
\begin{equation}\label{eq: eq before limit}
H(\zeta) =  G_{\zeta}*( u (H(\zeta)+1)) -\frac{i}{\zeta-\lambda}(G_{\zeta}-G_{\lambda})*\bigg(u\sum_{j=1}^N\phi_j\bigg).
\end{equation}
To obtain \eqref{eq: eq before limit} we have used $\phi_j = T\phi_j = G_{\lambda}*(u\phi_j)$. Take the limit as $\zeta\to\lambda$ to get
\begin{equation}\label{eq: limit eq}
H(\lambda) =  G_{\lambda}*\bigg(u\big( H(\lambda) +1 \big)\bigg) +x\sum_{j=1}^N\phi_j - G_{\lambda}*\bigg(u\big(x\sum_{j=1}^N\phi_j\big)\bigg) +\frac{i}{2\pi \lambda}\int_{\mathbb{R}}u\sum_{j=1}^N\phi_j~dx.
\end{equation}
\eqref{eq: limit eq} can be justified as follows. The left hand side of \eqref{eq: eq before limit} certainly converges in $L^2$ since $H(\zeta)$ is holomorphic from $\mathbb{C}$ to $L^2$. It follows that a subsequence converges almost everywhere. Furthermore, it is easy to see that $G_{\zeta}\to G_{\lambda}$ in $L^2$, hence the first term on the right hand side of \eqref{eq: eq before limit} converges uniformly by the fact that $u\in L^2\cap L^{\infty}$. Furthermore, it is not difficult to see that $\frac{G_{\zeta}-G_{\lambda}}{\zeta-\lambda}$ converges in $L^2$ to the following
\begin{equation}
\tilde{G}_{\lambda}(x)=\frac{1}{2\pi}\int_0^{\infty}e^{ix\xi}\frac{1}{(\xi-\lambda)^2}~d\xi.
\end{equation}
Therefore the second term on the right hand side of \eqref{eq: eq before limit} converges uniformly to \begin{equation}-i\tilde{G}_{\lambda}*\left(u\sum_{j=1}^N\phi_j\right).\end{equation} An integration by parts shows \begin{equation}\tilde{G}_{\lambda}(x)=-\frac{1}{2\pi \lambda}+ixG_{\lambda}(x),\end{equation} from which follows the other terms in \eqref{eq: limit eq}.
By the normalization of $\phi_j$ we have
\begin{equation}\label{eq: H lambda}H(\lambda) - x\sum_{j=1}^N\phi_j - G_{\lambda}*\bigg(u\big( H(\lambda)+1 - x\sum_{j=1}^N\phi_j\big)\bigg) =-N.\end{equation}
Equation \eqref{eq: H lambda} can be rewritten as
\begin{equation}\label{eq: V sat eq}
V-G_{\lambda}*(uV) = (1-N)G_{\lambda}*u,
\end{equation}
where 
\begin{equation}\label{eq: V eq}
V = H(\lambda) - x\sum_{j=1}^N\phi_j + N.
\end{equation}
We claim that $V\in L^2$. In fact, $H\in L^2$, while  $\phi_j$ is an eigenfunction, therefore $\phi_j = T\phi_j=G_{\lambda}*(u\phi_j)$. Now $ - x\sum_{j=1}^N\phi_j + N =- x\sum_{j=1}^NG_{\lambda}*(u\phi_j) + N\in L^2$ by Lemma \ref{lem: asymp} and the normalization of $\phi_j$. Recall that $I - G_{\lambda}*(u ~\cdot)$ is denoted by $I-T$. We have
\begin{equation}(1-N)G_{\lambda}*(u)\in \text{Ran}(I-T)\subset \overline{\text{Ran}(I-T)} = \text{Ker}(I-T^*)^{\perp}.\end{equation}
Evidently $T^* = uG_{\lambda}*~\cdot$. It is not hard to see that $\text{Ker}(I-T^*)\supset \text{Span}_{j=1}^N\{u\phi_j\}$. In fact,
\begin{equation}
T^*(u\phi_j) = uG_{\lambda}*(u\phi_j) = uT\phi_j =u\phi_j.
\end{equation}
If $1-N\ne 0$, we must have $G_{\lambda}*(u)\perp u\phi_j$ for every $j$. However 
\begin{align}
&\int_{\R}G_{\lambda}*u(x)\overline{u(x)\phi_j(x)}~dx\notag\\
= &\int_{\R^2}G_{\lambda}(x-y)u(y)\overline{u(x)\phi_j(x)}~dx~dy\notag\\
= &\int_{\R^2}u(y)\overline{G_{\lambda}(y-x)u(x)\phi_j(x)}~dx~dy\notag\\
= &\int_{\R}u(y)\overline{G_{\lambda}*(u\phi_j)(y)}~dy\notag\\
= &\int_{\R}u(y)\overline{\phi_j(y)}~dy\notag\\
= &-2\pi i \lambda \ne 0
\end{align}
by the normalization of $\phi_j$. Hence $1-N$ must be 0.
\end{proof}

As is mentioned before, Proposition \ref{prop: simplicity} is useful in the construction of scattering data in the Fokas and Ablowitz IST (see \cite{fokas1983inverse}). Apart from the location of the eigenvalues, another important piece of scattering data in the IST is the phase constant $\gamma$ in the following Corollary. It serves a similar role as the norming constants in the scattering data for Schr\"odinger operators. We include a discussion of the phase constant here since it follows as an easy consequence of the computations involved in Proposition \ref{prop: simplicity}. 
\begin{cor}
Let $\lambda$ be a negative eigenvalue of $L_u$, $\phi$ be the eigenvector normalized as \eqref{eq: normal} and let $H(\lambda)$ be defined as in \eqref{eq: limit eq}. Then there is a constant $\gamma$ so that
\begin{equation}
H(\lambda)+1 = (x+\gamma)\phi(x).
\end{equation}
\end{cor}
\begin{proof}
Letting $N=1$ in \eqref{eq: V sat eq}, we see that $V\in \text{Ker }(I-T)$. By Lemma \ref{lem: equiv diff int eq}, this implies that $V$ is an eigenvector of $L_u$ with eigenvalue $\lambda$. By Proposition \ref{prop: simplicity} there is a $\gamma$ so that \begin{equation}\label{eq: V} V = \gamma \phi.\end{equation} Plug \eqref{eq: V} into \eqref{eq: V eq}  and we get the desired result.
\end{proof}

\section{Finiteness of the Discrete Spectrum}\label{sec: fin}

In this section, we establish the fact that $L_u$ has only finitely many negative eigenvalues. In all of the following, $u$ is assumed to satisfy the conditions in Theorem \ref{thm: 2}.

Let us consider the operators 
\begin{equation}\label{def: L lambda u}
L_{\lambda u} = \frac{1}{i}\partial_x - \lambda C_+ u C_+.
\end{equation}
For any $\lambda\in \mathbb{R}$, \eqref{def: L lambda u} is a relatively compact perturbation of $\frac{1}{i}\partial_x$, and is bounded from below. It has essential spectrum $[0,\infty)$ and negative eigenvalues that could only accumulate at 0.
Define \begin{equation}\mu_n(\lambda) = \sup_{\varphi_1,\dots,\varphi_{n-1}}U_{\lambda}(\varphi_1,\dots,\varphi_{n-1}),\end{equation}
where
\begin{equation}U_{\lambda}(\varphi_1,\dots,\varphi_m) = \inf_{\psi \in H^+\cap H^1; \|\psi\|=1,\psi\in [\varphi_1,\dots,\varphi_m]^{\perp}}(\psi,L_{\lambda u}\psi).\end{equation}

One has the max-min principle (See \cite{reed1978methods}, Theorem XIII.1):
\begin{prop}
For each fixed $n$, either
\begin{enumerate}
\item $L_{\lambda u}$ has at least $n$ eigenvalues, and $\mu_n(\lambda)$ is the $n$th eigenvalue, 

or
\item $\mu_n(\lambda)=0$, and there are at most $n-1$ eigenvalues.
\end{enumerate}
\end{prop}

We remark here that $\mu_n(0)=0$ obviously. The asserted finiteness of the discrete spectrum of $L_u$ is equivalent to the statement that there are only finitely many $n$ for which $\mu_n(1)<0$. An immediate observation is that we only need to prove the result for positive $u$.

\begin{lemma}
Let $u^+$ be the positive part of $u$. If $L_{u+}$ has a finite discrete spectrum, so does $L_u$.
\end{lemma}
\begin{proof}
We denote the corresponding max-min expression for $L_{u^+}$ by $\mu_n^+(1)$. By its definition it is obvious that $\mu_n^+(1)\le \mu_n(1)$. Therefore if $\mu_N^+(1)=0$ for some $N$, $\mu_N(1)\ge 0$. By the max-min principle, $\mu_N(1) = 0$, and the number of eigenvalues of $L_u$ is at most $N-1$.
\end{proof}

From now on we always assume $u\ge 0$.

\begin{lemma}\label{lem: montonicity}
For any given $n$, $\mu_n(\lambda)$ is continuous for $\lambda \in \mathbb{R}$, decreasing on $[0,\infty)$ and is strictly decreasing once it becomes negative.
\end{lemma}
\begin{proof}
To get continuity of the $\mu_n$'s, one only needs to observe
\begin{equation}|(\psi,L_{\lambda u}\psi)-(\psi,L_{\lambda ' u}\psi)|\le |\lambda-\lambda'|\|u\|_{\infty}\|\psi\|_2^2.\end{equation}
For monotinicity, one just needs to see that
\begin{equation}(\psi,L_{\lambda u}\psi)-(\psi,L_{\lambda ' u}\psi) = (\lambda'-\lambda)(uC_+\psi,C_+\psi).\end{equation}
Let $\lambda>\lambda'>0$. If $\mu_n(\lambda')<0$, then $U_{\lambda'}(\varphi_1,\dots,\varphi_n)\le \mu_n(\lambda')<0$. For every choice of $\varphi_1,\dots,\varphi_n$, denote $S$ the set of allowed $\psi$'s. There exists a $\tilde{\psi}\in S$ so that \begin{equation}(\tilde{\psi},L_{\lambda ' u}\tilde{\psi})\le \inf_{\psi\in S}(\psi,L_{\lambda ' u}\psi)+\epsilon,\end{equation} where $\epsilon < \frac{\lambda-\lambda'}{2\lambda}(-\mu_n(\lambda'))$ is a chosen positive number. Notice that \begin{equation}-\lambda'(uC_+\tilde{\psi},C_+\tilde{\psi})\le (\tilde{\psi},L_{\lambda ' u}\tilde{\psi}),\end{equation}
therefore
\begin{equation}(uC_+\tilde{\psi},C_+\tilde{\psi})\ge -\frac{1}{\lambda'}(\tilde{\psi},L_{\lambda ' u}\tilde{\psi})\ge -\frac{1}{\lambda'}(\inf_{\psi\in S}(\psi,L_{\lambda ' u}\psi)+\epsilon)\ge -\frac{1}{\lambda'}(\mu_n(\lambda')+\epsilon).\end{equation}
Hence we have 
\begin{align}
&\inf_{\psi \in S}(\psi,L_{\lambda u}\psi)-\inf_{\psi \in S}(\psi,L_{\lambda ' u}\psi)\notag\\
\le &(\tilde{\psi},L_{\lambda u}\tilde{\psi})-(\tilde{\psi},L_{\lambda ' u}\tilde{\psi})+\epsilon\notag\\
\le &(\lambda'-\lambda)(uC_+\tilde{\psi},C_+\tilde{\psi})+\epsilon\notag\\
\le &\frac{\lambda-\lambda'}{\lambda'}(\mu_n(\lambda')+\epsilon)+\epsilon\notag\\
\le &\frac{\lambda-\lambda'}{\lambda'}\left(\mu_n(\lambda')+\frac{\lambda}{\lambda-\lambda'}\epsilon\right)\notag\\
\le &\frac{\lambda-\lambda'}{2\lambda'}\mu_n(\lambda'). \label{ineq: last}
\end{align}
The last inequality \eqref{ineq: last} follows from the choice of $\epsilon$. Therefore \begin{equation}U_{\lambda}(\varphi_1,\dots,\varphi_n)\le U_{\lambda'}(\varphi_1,\dots,\varphi_n)- \frac{\lambda-\lambda'}{2\lambda'}(-\mu_n(\lambda'))\end{equation} for every choice of $\varphi_1, \dots, \varphi_n$, and so
\begin{equation}\mu_n(\lambda)\le \mu_n(\lambda')- \frac{\lambda-\lambda'}{2\lambda'}(-\mu_n(\lambda')).\end{equation}
In other words, $\mu_n(\lambda)$ is strictly decreasing on $(0,\infty)$ once it becomes negative.
\end{proof}

By Lemma \ref{lem: montonicity}, for any $-E<0$, $\mu_n(1)<-E$ if and only if $\mu_n(\lambda)=-E$ for some $\lambda \in (0,1)$. By the max-min principle $-E$ is an eigenvalue of $L_{\lambda u}$:
\begin{equation}\frac{1}{i}\label{eq: e-value switch} \varphi_x -\lambda C_+(u\varphi) = -E\varphi, \quad \text{for some }\varphi \in D(L_{\lambda u}).\end{equation}
By Lemma \ref{lem: equiv diff int eq}, \eqref{eq: e-value switch} is equivalent to
\begin{equation}\varphi = \lambda G_{-E}*(u\varphi)\end{equation}
or
\begin{equation}\frac{1}{\lambda}\sqrt{u}\varphi = \sqrt{u}G_{-E}*(\sqrt{u}\cdot \sqrt{u}\varphi).\end{equation}
In other words, $\frac{1}{\lambda}$ is an eigenvalue of the operator 
\begin{equation}
K_{-E} = \sqrt{u}G_{-E}*(\sqrt{u}~\cdot)
\end{equation}
defined on $L^2$. The argument in Section \ref{sec: simp} shows that the negative eigenvalues of $L_{\lambda u}$ are all simple, therefore the $\mu_n(\lambda)$'s don't split or cross, and different $n$ gives rise to different $\frac{1}{\lambda}$ as eigenvalues of $K_{-E}$. Therefore to prove finiteness of the discrete spectrum of $L_u$, we just need to give a uniform bound on the number of eigenvalues greater than 1 of $K_{-E}$ as $E$ approaches 0. One way to show this, as is done for the Birman-Schwinger bound of Schr\"odinger operators, is to show that the Hilbert-Schmidt norm of $K_{-E}$ is uniformly bounded as $E$ approaches 0 (see \cite{reed1978methods}). Unfortunately it is not the case here. It turns out that the Hilbert-Schmidt norm of $K_{-E}$ blows up as $E$ approaches 0. What is actually happening is that one of the eigenvalues grows without bound as $E$ approaches 0. If one carefully peels off the effect of this single eigenvalue, the rest is seen to be bounded. The situation very much resembles the case for one and two dimensional Schr\"odinger operators, which was studied in \cite{simon1976bound}\footnote{The author would like to thank Peter Perry of University of Kentucky for pointing this out.}.

We write
\begin{equation}G_{-E}(x) = \int_{0}^{\infty}\frac{e^{ix\xi}}{\xi+E}~d\xi = \int_0^{\infty}\frac{e^{ix\xi}-\chi(\xi)}{\xi+E}~d\xi + \int_0^{\infty}\frac{\chi(\xi)}{\xi +E}~d\xi = N_{-E}(x)+R(E).\end{equation}
Here $\chi(\xi)$ is a compactly supported non-negative smooth cutoff function which is identically equal to 1 when $|\xi|\le 1$, has support on $|\xi|\le 2$, and is always between 0 and 1. Then
\begin{equation}
K_{-E} = M_{-E} + L_{-E},
\end{equation}
where $M_{-E} = \sqrt{u}N_{-E}*(\sqrt{u}~\cdot)$ and $L_{-E} = \sqrt{u}R(E)(\sqrt{u},\cdot)$. 

\begin{lemma}\label{lem: HS bound}
$M_{-E}$ is Hilbert-Schmidt for every $E\ge 0$ and for $E$ small enough $\|M_{-E}-M_{0}\|_{HS}\le CE^{\epsilon}$ for some $C>0$ and $\epsilon>0$.
\end{lemma}
\begin{proof}
The computation of the Hilbert-Schmidt norm of $M_{-E}$ as well as $M_{-E}-M_0$ will depend on some pointwise estimates of $N_{-E}(x)$. In fact, we will first find a pointwise estimate for $N_0(x)$, and use it to show that $M_0$ has bounded Hilbert-Schmidt norm. Following that we will find a pointwise estimate for $N_{-E}(x) - N_0(x)$ and show boundedness of the Hilbert Schmidt norm of $M_{-E}-M_0$. Let us write $N_0(x)=\int_0^{\infty}\frac{e^{ix\xi}-\chi(\xi)}{\xi}~d\xi$ as
\begin{equation}\int_0^1 \frac{e^{ix\xi}-1}{\xi}~d\xi + \int_1^2 \frac{e^{ix\xi}-\chi(\xi)}{\xi}~d\xi+\int_2^{\infty} \frac{e^{ix\xi}}{\xi}~d\xi =I + II + III.\end{equation}
We estimate the three terms one by one.
\begin{equation}\label{eq: int split} I = \int_0^x\frac{e^{i\xi}-1}{\xi}~d\xi = \int_0^{\pm 1}\frac{e^{i\xi}-1}{\xi}~d\xi + \int_{\pm 1}^x\frac{e^{i\xi}-1}{\xi}~d\xi\end{equation}
The choice of the $\pm$ sign in front of 1 is to agree with the sign of $x$. The first integral in \eqref{eq: int split} is finite and the second integral is bounded by
\begin{equation}\left|\int_1^{|x|}\frac{2}{\xi}~d\xi\right|=2|\log |x||.\end{equation}
$II$ is easily seen to be bounded since $|e^{ix\xi}-\chi(\xi)|\le 2$.
We will only show estimate for $III$ when $x>0$. The case when $x<0$ is similar.
\begin{equation}III = \frac{1}{ix}\frac{e^{ix\xi}}{\xi}\bigg|_2^{\infty} +\frac{1}{ix}\int_2^{\infty}\frac{e^{ix\xi}}{\xi^2}~d\xi\end{equation}
is apparently bounded when $x>1$. When $0<x\le1$, write $III$ as
\begin{equation}\label{eq: int split 2} \int_{2x}^{\infty}\frac{e^{i\xi}}{\xi}~d\xi = \int_{2}^{\infty}\frac{e^{i\xi}}{\xi}~d\xi + \int_{2x}^2\frac{e^{i\xi}}{\xi}~d\xi, \end{equation}
whereas the first term in \eqref{eq: int split 2} is finite and the second term is bounded by 
\begin{equation}\int_{2x}^2 \frac{1}{\xi}~d\xi = |\log |x||.\end{equation}
In summary, we have shown that $|N_0(x)|\le C(1+|\log |x||)$ for all $x\ne 0$ for some uniform constant $C$, therefore $|N_0(x)|^2 \le C(1+(\log|x|)^2)$ for some other constant $C$. Let us now compute the Hilbert-Schmidt norm of $M_0$:
\begin{align*}
&\|M_0\|_{HS}^2\\
= &\int_{\R^2}u(x)u(y)|N_0(x-y)|^2~dx~dy\\
\le &C\int_{\R^2}u(x)u(y)(1+(\log|x-y|)^2)~dx~dy.\numberthis
\end{align*} 
Since $u\in L^1$, we only need to compute
\begin{align*}
&\int_{\R^2}u(x)u(y)(\log|x-y|)^2~dx~dy\\
= &\int_{\R^2}(1+|x|)u(x)(1+|y|)u(y)\frac{(\log|x-y|)^2}{(1+|x|)(1+|y|)}~dx~dy. \numberthis \label{eq: int finite}
\end{align*}
Since $u(x)$ and $xu(x)$ are both in $L^2$, we see that in order for \eqref{eq: int finite} to be finite, one just needs
\begin{equation}\int_{\R^2}\frac{(\log|x-y|)^4}{(1+|x|)^2(1+|y|)^2}~dx~dy\end{equation}
to be finite. We break up this integral into two pieces: one on $|x-y|\le 1$ and one on $|x-y|\ge 1$.
\begin{align*}
&\int_{|x-y|\le 1}\frac{(\log|x-y|)^4}{(1+|x|)^2(1+|y|)^2}~dx~dy\\
= &\int_{|z|\le 1}\frac{(\log|z|)^4}{(1+|x|)^2(1+|x-z|)^2}~dx~dz\\
\le &\int_{-\infty}^{\infty} \frac{1}{(1+|x|)^2}\int_{-1}^1(\log|z|)^4~dz~dx\\
< &\infty. \numberthis
\end{align*}
\begin{align*}
&\int_{|x-y|\ge 1}\frac{(\log|x-y|)^4}{(1+|x|)^2(1+|y|)^2}~dx~dy\\
\le &\int_{|x-y|\ge 1}\frac{(\log(1+|x-y|))^4}{(1+|x|)^2(1+|y|)^2}~dx~dy\\
\le &\int_{|x-y|\ge 1}\frac{(\log(1+|x|)+\log(1+|y|))^4}{(1+|x|)^2(1+|y|)^2}~dx~dy\\
\le &C\int_{\R^2}\frac{(\log(1+|x|))^4+(\log(1+|y|))^4}{(1+|x|)^2(1+|y|)^2}~dx~dy\\
< &\infty, \numberthis
\end{align*}
where we have used the elementary inequality $1+|x-y|\le (1+|x|)(1+|y|)$.
Next we find a similar estimate on $M_{-E}-M_0$. In fact,
\begin{align*}
&N_{-E}(x)-N_0(x) \numberthis\\
= &\int_0^1 (e^{ix\xi}-1)\left(\frac{1}{\xi+E}-\frac{1}{\xi}\right)~d\xi + \int_1^2 (e^{ix\xi}-\chi(\xi))\left(\frac{1}{\xi+E}-\frac{1}{\xi}\right)~d\xi+\int_2^{\infty} e^{ix\xi}\left(\frac{1}{\xi+E}-\frac{1}{\xi}\right)~d\xi\\
= &I + II + III. 
\end{align*}
In the following estimates, we work with the assumption $x>0$, and the $x<0$ case will be similar.
\begin{equation}I = \int_0^x  (e^{i\xi}-1)\left(\frac{1}{\xi+Ex}-\frac{1}{\xi}\right)~d\xi.\end{equation}
When $0<x<1$, $I$ is bounded by
\begin{equation}\int_0^x \xi\left(\frac{1}{\xi}-\frac{1}{\xi+Ex}\right)~d\xi = Ex\log\frac{1+E}{E}\le C(E+E^{\epsilon})\end{equation}
for some uniform constant $C$ and some small $\epsilon>0$.
When $x>1$,
\begin{equation}I = \int_0^1(e^{i\xi}-1)\left(\frac{1}{\xi+Ex}-\frac{1}{\xi}\right)~d\xi + \int_1^x(e^{i\xi}-1)\left(\frac{1}{\xi+Ex}-\frac{1}{\xi}\right)~d\xi\end{equation}
The first term is bounded as before, and the second term is bounded by
\begin{equation}\int_1^x 2\left(\frac{1}{\xi}-\frac{1}{\xi+Ex}\right)~d\xi=2\log\left(\frac{1+Ex}{1+E}\right)\le C(E^{\epsilon}|x|^{\epsilon}+E)\end{equation}
for some small $\epsilon>0$.
$II$ is bounded by
\begin{equation}\int_1^2 2\left(\frac{1}{\xi}-\frac{1}{\xi+E}\right)~d\xi = 2\log\frac{2+2E}{2+E}\le CE.\end{equation}
$III$ is bounded by
\begin{equation}\int_2^{\infty}\left(\frac{1}{\xi}-\frac{1}{\xi+E}\right)~d\xi = \log\frac{2+E}{2}\le CE.\end{equation}
In summary, we have shown that $|N_{-E}(x)-N_0(x)|\le C(E+E^{\epsilon}(1+|x|^{\epsilon}))$.
The calculation of the Hilbert-Schmidt norm of $M_{-E}-M_0$ is similar as before. One just needs to let $|x|^{\epsilon}$ with a small $\epsilon$ replace $|\log |x||$.
\end{proof}

From now on we assume $0<E<E_0$ and $0<\lambda<\lambda_0$. The size of $E_0$ and $\lambda_0$ will be determined in the following argument. 

\begin{lemma}\label{lem: eigen eq}
$\frac{1}{\lambda}$ is an eigenvalue of $K_{-E}$ if and only if 
\begin{equation}\label{eq: sub1}
\lambda R(E)(\sqrt{u},(1-\lambda M_{-E})^{-1}\sqrt{u})=1.
\end{equation}
\end{lemma}
\begin{proof}
This argument is carried over essentially unchanged from \cite{simon1976bound}. If $\frac{1}{\lambda}$ is an eigenvalue of $K_{-E}$, then $\lambda K_{-E}$ has 1 as an eigenvalue and therefore
\begin{equation}\text{det}_2(1-\lambda K_{-E})=0.\end{equation}
Here det$_2$ for a Hilbert-Schmidt operator $A$ is defined as in \cite{simon1976bound}:
\begin{equation}\text{det}_2(1+A) = \text{det}((1+A)e^{-A}).\end{equation}
It follows from 
\begin{equation}1-\lambda K_{-E} = 1-\lambda M_{-E} - \lambda L_{-E} = (1-\lambda M_{-E})\big(1-(1-\lambda M_{-E})^{-1}\lambda L_{-E}\big)\end{equation}
that
\begin{equation}\text{det}\big(1-(1-\lambda M_{-E})^{-1}\lambda L_{-E}\big)=0.\end{equation}
The invertibility of $1-\lambda M_{-E}$ for small $\lambda$ and $E$ follows from Lemma \ref{lem: HS bound}. Here of course we assume $\lambda_0$ has been chosen small in order to make the norm of $\lambda M_{-E}$ small. Since $L_{-E}$ has rank one, 
\begin{equation}\text{tr}\big((1-\lambda M_{-E})^{-1}\lambda L_{-E}\big)=1\end{equation}
or
\begin{equation}\lambda R(E)(\sqrt{u},(1-\lambda M_{-E})^{-1}\sqrt{u})=1.\end{equation}
\end{proof}

\begin{lemma}\label{lem: unique sol lambda}
There are $E_0>0$ and $\lambda_0>0$ such that for each $E\in (0, E_0)$, there exists a unique $\lambda \in (0,\lambda_0]$ solving \eqref{eq: sub1}.
\end{lemma}
\begin{proof}
Rewrite \eqref{eq: sub1} as
\begin{equation}\frac{1}{R(E)} = \lambda (\sqrt{u},\sqrt{u}) + \lambda ^2 (\sqrt{u},M_{-E}\sqrt{u}) + \lambda^3 B(\lambda,E)\end{equation} 
and furthermore as
\begin{equation}\label{eq: sub2}
\lambda = \frac{1}{\int u}\bigg(\frac{1}{R(E)} -\lambda ^2 (\sqrt{u},M_{-E}\sqrt{u}) - \lambda^3 B(\lambda,E)\bigg)
\end{equation}
where $B(\lambda, E) = \sum_{n=0}^{\infty}\lambda^n(\sqrt{u},M_{-E}^{n+2}\sqrt{u})$. Here we have assumed that $u$ is not indentically zero, hence $\int u >0$. 
Obviously $R(E)\to +\infty$ as $E\to 0$. Hence by choosing $E_0, \lambda_0$ small, and fixing $E\in (0,E_0)$, the right hand side of \eqref{eq: sub2} is a continuous contraction from $[0,\lambda_0]$ to itself, therefore has a unique fixed point. Finally $\lambda=0$ is evidently not the solution.
\end{proof}

From Lemma \ref{lem: eigen eq} and Lemma \ref{lem: unique sol lambda} we see that for $E<E_0$, $K_{-E}$ has exactly one eigenvalue in $\left[\frac{1}{\lambda_0},\infty\right)$. We estimate the size of $\frac{1}{\lambda}$ as follows. From \eqref{eq: sub2} we get
\begin{equation}\frac{1}{\lambda} = R(E)\bigg(\int u +\lambda(\sqrt{u},M_{-E}\sqrt{u})+O(\lambda^2)\bigg).\end{equation}
Therefore
\begin{equation}\lambda R(E) = \frac{1}{\int u}(1+O(\lambda)).\end{equation}
Hence
\begin{align}
\frac{1}{\lambda^2} = &R^2(E)\bigg(\int u + \lambda(\sqrt{u},M_{-E}\sqrt{u}) + O(\lambda^2)\bigg)^2 \notag \\
=& R^2(E) \bigg(\int u\bigg)^2 + 2R^2(E)\lambda \bigg(\int u\bigg)(\sqrt{u}, M_{-E}\sqrt{u}) + O(1)\notag \\
= &R^2(E) \bigg(\int u\bigg)^2 + 2R(E)(\sqrt{u}, M_{-E}\sqrt{u}) + O(1).\label{eq: 3}
\end{align}

We are now ready to prove
\begin{prop}\label{prop: bounded no eigen}
For $E\in (0,E_0)$, the number of eigenvalues of $K_{-E}$ that are greater than 1 is uniformly bounded.
\end{prop}
\begin{proof}
Since $K_{-E}$ is clearly self-adjoint, one has $\sum_{\mu \text{ eigenvalue of }K_{-E}}\mu^2 = \|K_{-E}\|_{HS}^2$. By Lemma \ref{lem: eigen eq} and Lemma \ref{lem: unique sol lambda}, all of the eigenvalues of $K_{-E}$ except one are less than $\frac{1}{\lambda_0}$. Hence the number of eigenvalues greater than 1 is bounded by
\begin{equation}1+ \bigg(\sum_{\mu \text{ eigenvalue of }K_{-E}}\mu^2\bigg) - \frac{1}{\lambda^2} = 1+\|K_{-E}\|_{HS}^2 - \frac{1}{\lambda^2},\end{equation}
where $\frac{1}{\lambda}$ is the unique eigenvalue exceeding $\frac{1}{\lambda_0}$.
By \eqref{eq: 3}
\begin{align*}
&\|K_{-E}\|_{HS}^2- \frac{1}{\lambda^2}\numberthis\\
= ~& \|K_{-E}\|_{HS}^2 - R^2(E) \bigg(\int u\bigg)^2 - 2R(E)(\sqrt{u}, M_{-E}\sqrt{u}) + O(1)\\
= ~& \int_{\mathbb{R}^2}u(x)u(y)|G_{-E}(x-y)|^2~dx~dy - R^2(E) \bigg(\int u\bigg)^2 - 2R(E)(\sqrt{u}, \sqrt{u}N_{-E}*(\sqrt{u}\sqrt{u})) + O(1)\\
= ~& \int_{\mathbb{R}^2}u(x)u(y)R(E)^2~dx~dy + 2R(E)\int_{\mathbb{R}^2}u(x)u(y)N_{-E}(x-y)~dx~dy \\
 & + \int_{\mathbb{R}^2}u(x)u(y)|N_{-E}(x-y)|^2~dx~dy- R^2(E) \bigg(\int u\bigg)^2 \\
 & - 2R(E)\int_{\mathbb{R}^2}u(x)u(y)N_{-E}(x-y)~dx~dy + O(1)\\
= ~&\int_{\mathbb{R}^2}u(x)u(y)|N_{-E}(x-y)|^2~dx~dy + O(1)\\
= ~&O(1). 
\end{align*}
In the midst of these calculations, we have used $G_{-E}(x)=R(E) + N_{-E}(x)$ and the fact that
$\overline{N_{-E}(x)} = N_{-E}(-x)$.
For the last step, we need Lemma \ref{lem: HS bound}.
\end{proof}

By the discussion above, the number of eigenvalues greater than $1$ of $K_{-E}$ bounds the number of eigenvalues of $L_u$ that are greater than $-E$. Proposition \ref{prop: bounded no eigen} therefore implies that the discrete spectrum of $L_u$ is finite.


\bibliography{BObiblio}   

\end{document}